\UseRawInputEncoding
\documentclass{amsart}

\usepackage{xcolor,amssymb,stmaryrd}

\newtheorem{theorem}{Theorem}[section]
\newtheorem{lemma}[theorem]{Lemma}

\newtheorem{corollary}[theorem]{Corollary}

\theoremstyle{definition}

\theoremstyle{remark}
\newtheorem{remark}[theorem]{Remark}

\usepackage[pagebackref,colorlinks=true]{hyperref}  

\newcommand{\ie}{\mbox{i.\hspace{.5pt}e.}\ }
\newcommand{\eg}{\mbox{e.\hspace{.5pt}g.}\ }
\newcommand{\f}{\varphi}
\newcommand{\g}{\tilde{g}}
\newcommand{\n}{\nabla}
\newcommand{\tn}{\tilde{\n}}
\newcommand{\ttt}{\tilde\tau}
\newcommand{\MM}{\mathcal{M}}
\newcommand{\NN}{\mathcal{N}}
\newcommand{\M}{(\MM,\A\f,\A\xi,\A\eta,\A{}g)}

\newcommand{\I}{\iota}

\newcommand{\R}{\mathbb R}

\newcommand{\F}{\mathcal{F}}
\newcommand{\LL}{\mathcal{L}}
\newcommand{\ta}{\theta}
\newcommand{\om}{\omega}
\newcommand{\lm}{\lambda}
\newcommand{\gm}{\gamma}

\newcommand{\vt}{\vartheta}
\newcommand{\ttau}{\tilde{\tau}}
\newcommand{\tlm}{\tilde{\lm}}
\newcommand{\tvt}{\tilde{\vt}}
\newcommand{\tgm}{\tilde{\gm}}
\newcommand{\tf}{\tilde{f}}
\newcommand{\tk}{\tilde{k}}
\newcommand{\tF}{\tilde{F}}
\newcommand{\h}{\tilde{h}}

\newcommand{\D}{\mathrm{d}\hspace{-0.5pt}}

\newcommand{\ddt}{\tfrac{\D}{\D t}}

\DeclareMathOperator{\grad}{grad} 

\newcommand{\A}{\allowbreak{}}

\newcommand{\thmref}[1]{Theorem~\ref{#1}}

\newcommand{\corref}[1]{Corollary~\ref{#1}}

\newcommand{\remref}[1]{Remark~\ref{#1}}

\begin{document}

\title[
Pairs of associated Yamabe almost  solitons ...]
{Pairs of associated Yamabe almost solitons with vertical potential
on almost contact complex Riemannian manifolds } 


\author{Mancho Manev}
\address[MM1]
{University of Plovdiv Paisii Hilendarski,
Faculty of Mathematics and Informatics, Department of Algebra and
Geometry, 24 Tzar Asen St., Plovdiv 4000, Bulgaria
}

\email{mmanev@uni-plovdiv.bg}

\address[MM2]
{
Medical University -- Plovdiv,
Faculty of Pharmacy,
Department of Medical Physics and Biophysics,
15A Vasil Aprilov Blvd,
Plovdiv 4002, Bulgaria}

\email{mancho.manev@mu-plovdiv.bg}




\begin{abstract}
Almost contact complex Riemannian manifolds, known also as almost contact B-metric manifolds,
are in principle equipped with a pair of mutually associated pseudo-Riemannian metrics.
Each of these metrics is specialized here as a Yamabe  almost soliton with a potential collinear to the Reeb vector field.
The resulting manifolds are then investigated in two important cases with geometric significance.
The first is when the manifold is of Sasaki-like type,
\ie its complex cone is a holomorphic complex Riemannian
manifold (also called a K\"ahler--Norden manifold).
The second case is when the soliton potential is torse-forming,
\ie it satisfies a certain recurrence condition for its covariant derivative with respect to the
Levi-Civita connection of the corresponding metric.
The studied solitons are characterized.
In the three-dimensional case, an explicit example is constructed and the properties obtained in the theoretical part are confirmed.
\end{abstract}



%
%


\maketitle

\section{Introduction}

The notion of the Yamabe flow is known since 1988. It is introduced in \cite{Ham88,Ham89} by R.\,S.\ Hamilton to construct metrics with constant scalar curvature.

A time-dependent family of (pseudo-)Riemannian metrics $g(t)$ considered on a smooth manifold $\MM$ is said to evolve by \emph{Yamabe flow} if $g(t)$ satisfies the following evolution equation
\[
\frac{\partial}{\partial t} g(t) = -\tau(t) g(t),\qquad g(0)=g_0,
\]
where $\tau(t)$ denotes the scalar curvature corresponding to $g(t)$.

A self-similar solution of the Yamabe flow on $(\MM, g)$ is called a \emph{Yamabe soliton} and is determined by the following equation
\begin{equation}\label{YS-g}
  \frac12 \LL_{\vt} g = (\tau - \lm) g,
\end{equation}
where $\LL_{\vt} g$ denotes the Lie derivative of $g$ along the vector field $\vt$ called the soliton potential, and $\lm$ is the soliton constant (\eg \cite{BarRib}).
Briefly, we denote this soliton by $(g;\vt,\lm)$.
In the case that $\lm$ is a differential function on $\MM$, the solution is  called an \emph{Yamabe almost  soliton}.

Many authors have studied Yamabe (almost) solitons on different types of manifolds in the recent years (see \eg \cite{Cao1,Chen1,Das1,Gho1,Roy1,Man73,Man75}). The study of this kind of flows and the corresponding (almost) solitons cause an interest in mathematical physics because the Yamabe flow corresponds to fast diffusion of the porous medium equation \cite{ChLuNi}.

In \cite{Man73} the author begins the study of Yamabe solitons on almost contact complex Riemannian manifolds (abbreviated accR manifolds), there called almost contact B-metric manifolds.
These manifolds are classified in \cite{GaMiGr} by G.\ Ganchev, V.\ Mihova and K.\ Gribachev.

The pair of B-metrics, which are related to each other by the almost contact structure, determine the geometry of the investigated manifolds.
In \cite{Man73} and \cite{Man75}, the present author studies Yamabe solitons obtained by contact conformal transformations for some interesting classes of studied manifolds.
In the former paper, the manifold studied is cosymplectic or Sasaki-like, and in the latter, the soliton potential is torse-forming.
Contact conformal transformations of an almost contact B-metric structure
transform the two B-metrics, the Reeb vector field and its dual contact 1-form, using this pair of metrics and a triplet of differentiable functions on the manifold (see \eg \cite{ManGri93,Man4}).
These transformations generalize the $\mathcal{D}$-homothetic deformations of the considered manifolds introduced in \cite{Bulut19}.

In the present work, instead of these naturally occurring transformed Yamabe solitons involving the two B-metrics, we use a condition for two Yamabe almost solitons for each of the metrics.
Again, one of the simplest types of non-cosymplectic manifolds among those investigated, which is of interest to us, is precisely that the Sasaki-likes introduced in \cite{IvMaMa45}.
This means that a warped product of a Sasaki-like accR manifold with the positive real axis gives rise to a complex cone which is a K\"ahler manifold with a pair of Norden metrics.
Note that the intersection of the classes of Sasaki-like manifolds and cosymplectic manifolds is an empty set.
Different types of solitons on Sasaki-like manifolds have been studied in \cite{Man73,Man62,Man63,Man66}.

Another interesting type of the studied manifold with Yamabe solitons is again (as in \cite{Man73} and \cite{Man75}) the object of consideration in the present article. This is the case when the soliton potential is a torse-forming vertical vector field. Vertical means it has the same direction as the Reeb vector field.
Torse-forming vector fields are defined by a certain recurrence condition for their covariant derivative regarding the
Levi-Civita connection of the basic metric. These vector fields are first defined and studied by K.\ Yano \cite{Yano44}. They are then
investigated by various authors for manifolds with different tensor structures (\eg \cite{MiRoVe96, MihMih13, Chen17}) and for the studied here manifolds in \cite{Man75,Man62, Man66, Man67}.

The present paper is organized as follows.
After the present introduction to the topic, in
Section~2 we recall some known facts
about the investigated manifolds.
In Section~3, we set ourselves the task of equipping the considered manifolds with a pair of associated Yamabe almost solitons.
In Section~4, we prove that there does not exist a Sasaki-like manifold equipped with a pair of Yamabe almost solitons with vertical potential generated by each of the two fundamental metrics.
A successful solution to the problem posed in Section~3 is done in Section~5 in the case where the vertical potentials of the pair of Yamabe almost solitons are torse-forming.
Section~6 provides an explicit example of the smallest dimension of the type of manifold constructed in the previous section.


\section{accR manifolds}


A differentiable manifold $\MM$ of dimension $(2n+1)$, equipped with an almost contact structure $(\f,\xi,\eta)$
and a B-metric $g$ is called an \emph{almost contact B-metric man\-i\-fold} or \emph{almost contact complex Riemannian} (abbr.\ \emph{accR}) \emph{manifold} and it is denoted by $\M$.
More concretely, $\f$ is an endomorphism
of the tangent bundle $T\MM$, $\xi$ is a Reeb vector field, $\eta$ is its dual contact 1-form and
$g$ is a pseu\-do-Rie\-mannian
metric  of signature $(n+1,n)$ satisfying the following conditions
\begin{equation}\label{strM}
\begin{array}{c}
\f\xi = 0,\qquad \f^2 = -\I + \eta \otimes \xi,\qquad
\eta\circ\f=0,\qquad \eta(\xi)=1,\\[4pt]
g(\f x, \f y) = - g(x,y) + \eta(x)\eta(y),
\end{array}
\end{equation}
where $\I$ stands for the identity transformation on $\Gamma(T\MM)$ \cite{GaMiGr}.

In the latter equality and further, $x$, $y$, $z$ 
will stand for arbitrary elements of $\Gamma(T\MM)$ or vectors in the tangent space $T_p\MM$ of $\MM$ at an arbitrary
point $p$ in $\MM$.

The following equations are immediate consequences of \eqref{strM}
\[
g(\f x, y) = g(x,\f y),\qquad g(x, \xi) = \eta(x),\qquad
g(\xi, \xi) = 1,\qquad \eta(\n_x \xi) = 0,
\]
where $\n$ denotes the Levi-Civita connection of $g$.

The associated metric $\g$ of $g$ on $\MM$ is also a B-metric and it is defined by
\begin{equation}\label{gg}
\g(x,y)=g(x,\f y)+\eta(x)\eta(y).
\end{equation}

In  \cite{GaMiGr}, accR manifolds 
are classified with respect
to the (0,3)-tensor $F$ defined by
\begin{equation}\label{F=nfi}
F(x,y,z)=g\bigl( \left( \nabla_x \f \right)y,z\bigr).
\end{equation}
It has the following basic properties:
\begin{eqnarray}\label{F-prop1}
&F(x,y,z)=F(x,z,y)
=F(x,\f y,\f z)+\eta(y)F(x,\xi,z)
+\eta(z)F(x,y,\xi),\\[4pt] \label{F-prop2}
&F(x,\f y, \xi)=(\n_x\eta)y=g(\n_x\xi,y).
\end{eqnarray}
The Ganchev--Mihova--Gribachev classification of the studied manifolds cited in the introduction consists of eleven basic classes $\F_i$, $i\in\{1,2,\dots,11\}$, determined by conditions for $F$.

%



\section{Pair of associated Yamabe almost solitons}

Let us consider an accR manifold $\M$ with a pair of associated Yamabe almost  solitons generated by the pair of B-metrics $g$ and $\g$, \ie $(g;\vt,\lm)$ and $(\g;\tvt,\tlm)$, which are mutually associated by the $(\f,\xi,\eta)$-structure. Then, along with \eqref{YS-g}, the following identity also holds
\begin{equation}\label{YS-tg}
  \frac12 \LL_{\tvt} \g = (\ttau - \tlm) \g,
\end{equation}
where $\tvt$ and $\tlm$ are the soliton potential and the soliton function, respectively, and $\ttau$ is the scalar curvature of the manifold with respect to $\g$.
We suppose that the potentials $\vt$ and $\tvt$ are vertical, \ie there exists differentiable functions $k$ and $\tk$ on $\MM$,
such that we have
\begin{equation}\label{kk}
  \vt=k\xi,\qquad \tvt=\tk\xi,
\end{equation}
where $k(p)\neq 0$ and $\tk(p)\neq 0$ at every point $p$ of $M$. Briefly, we denote these potentials by $(\vt,k)$ and $(\tvt,\tk)$.

In this case, for the Lie derivatives of $g$ and $\g$ along $\vt$ and $\tvt$, respectively, we obtain the following expressions:
\begin{equation}\label{Lg}
\begin{split}
\left(\LL_{\vt}g\right)(x,y) &= g(\n_x \vt, y)+g(x, \n_y \vt)\\[4pt]
&=\D{k}(x)\eta(y) + \D{k}(y)\eta(x) +k \left\{g(\n_x \xi, y)+g(x, \n_y \xi)\right\},\\[4pt]
\end{split}
\end{equation}
\begin{equation}
\label{Ltg}
\begin{split}
\left(\LL_{\tvt}\g\right)(x,y) &= \g(\tn_x \tvt, y)+\g(x, \tn_y \tvt)\\[4pt]
&=\D{\tk}(x)\eta(y) + \D{\tk}(y)\eta(x) +\tk \left\{\g(\tn_x \xi, y)+\g(x, \tn_y \xi)\right\}.
\end{split}
\end{equation}

\section{The case when the underlying accR manifold is Sasaki-like}

In \cite{IvMaMa45}, it is introduced the type of a \emph{Sasaki-like} manifold among accR manifolds. The definition condition is its complex cone to be a K\"ahler-Norden manifold, \ie with a parallel complex structure.
A Sasaki-like accR manifold 
is determined by the condition
\[
\left(\nabla_x\f\right)y=\eta(y)\f^2 x+g(\f x,\f y)\xi.
\]
Therefore, the fundamental tensor $F$ of such a manifold has the following form
\begin{equation}\label{defSlF}
\begin{array}{l}
F(x,y,z)=g(\f x,\f y)\eta(z)+g(\f x,\f z)\eta(y).
\end{array}
\end{equation}

Obviously, Sasaki-like accR manifolds form a subclass of the class $\F_4$ of the Ganchev--Mihova--Gribachev classification. 
Moreover,
the following identities are valid for it
\begin{equation}\label{curSl}
\begin{array}{ll}
\n_x \xi=-\f x, \qquad &\left(\n_x \eta \right)(y)=-g(x,\f y),\\[4pt]
R(x,y)\xi=\eta(y)x-\eta(x)y, \qquad &\rho(x,\xi)=2n\, \eta(x), \\[4pt]
R(\xi,y)z=g(y,z)\xi-\eta(z)y,\qquad 				&\rho(\xi,\xi)=2n,
\end{array}
\end{equation}
where $R$ and $\rho$ stand for the curvature tensor and the Ricci tensor of $\n$, defined as usual by
$R=[\n,\n]-\n_{[\,,\,]}$ and $\rho$ is the result of the contraction of $R$ by its first index \cite{IvMaMa45}.


If the considered accR  manifold $\M$ is Sasaki-like, due to the first equality of \eqref{curSl}, we obtain that \eqref{Lg} takes the following form
\begin{equation}\label{Lg-Sl}
\begin{split}
\left(\LL_{\vt}g\right)(x,y)
=\D{k}(x)\eta(y) + \D{k}(y)\eta(x) -2k g(x, \f y).
\end{split}
\end{equation}

We then put the result of \eqref{Lg-Sl} into \eqref{YS-g} and get the following
\begin{equation}\label{YS-Sl}
\begin{split}
\frac12\left\{\D{k}(x)\eta(y) + \D{k}(y)\eta(x)\right\} -k g(x, \f y)
=(\tau - \lm) g(x,y).
\end{split}
\end{equation}
Replacing $x$ and $y$ with $\xi$ in \eqref{YS-Sl} gives
\begin{equation}\label{YS-Sl-Dk1}
\D{k}(\xi)=\tau - \lm.
\end{equation}
The trace of \eqref{YS-Sl} in an arbitrary  basis $\{e_i\}$ $(i=1,2,\dots,2n+1)$ implies
\begin{equation}\label{YS-Sl-Dk2}
\D{k}(\xi)=(2n+1)(\tau - \lm).
\end{equation}
Combining \eqref{YS-Sl-Dk1} and \eqref{YS-Sl-Dk2} leads to
$
k=0,
$ 
which contradicts the conditions and therefore we find the following to be true
\begin{theorem}
There does not exist a Sasaki-like manifold $\M$ equipped with a $g$-generated
Yamabe almost  soliton having a vertical potential.
\end{theorem}

Now, let us consider the similar situation but with respect to the associated B-metric $\g$ and the corresponding Levi-Civita connection $\tn$.

First, similarly to \eqref{F=nfi}, we define the fundamental tensor $\tF$ for $\g$ as follows
\begin{equation*}\label{tF}
\tF(x,y,z)=\g\left( \bigl( \tn_x \f \bigr)y,z\right).
\end{equation*}
Since $\g$ is also a B-metric like $g$, it is obvious that properties \eqref{F-prop1} and \eqref{F-prop2} also hold for $\tF$, \ie
\begin{eqnarray}\label{tF-prop1}\nonumber
&\tF(x,y,z)=\tF(x,z,y)
=\tF(x,\f y,\f z)+\eta(y)\tF(x,\xi,z)
+\eta(z)\tF(x,y,\xi),\\[4pt]\label{tF-prop2}
&\tF(x,\f y, \xi)=(\tn_x\eta)y=\g(\tn_x\xi,y).
\end{eqnarray}

Then, using the well-known Koszul formula in this case for $\g$, \ie
\[
\begin{split}
2\g\Big(\tn_{x}y,z\Big)=x\bigl(\g(y,z)\bigr)&+y\bigl(\g(x,z)\bigr)-z\bigl(\g(x,y)\bigr)\\
&+\g\bigl([x,y],z\bigr)+\g\bigl([z,y],x\bigr)+\g\bigl([z,x],y\bigr),
\end{split}
\]
we obtain, after lengthy but standard calculations, the following relationship between $\tF$ and $F$: \cite{ManA}
\begin{equation}\label{tFF}
\begin{split}
2\tF(x,y,z)=&\ F(\f y,z,x)-F(y, \f z,x)+F(\f z,y,x)-F(z, \f y,x)\\[4pt]
&+\left\{F(x,y,\xi)+F(\f y, \f x, \xi)+F(x,\f y,\xi)\right\}\eta(z)\\[4pt]
&+\left\{F(x,z,\xi)+F(\f z, \f x, \xi)+F(x,\f z,\xi)\right\}\eta(y)\\[4pt]
&+\left\{F(y,z,\xi)+F(\f z, \f y, \xi)+F(z,y,\xi)+F(\f y, \f z, \xi)\right\}\eta(x).
\end{split}
\end{equation}

\begin{lemma}
For a Sasaki-like manifold $\M$ with associated B-metric $\g$ the following holds
\begin{equation}\label{tnxi}
  \tn_x \xi=-\f x.
\end{equation}
\end{lemma}
\begin{proof}
Due to \eqref{tF-prop2} and \eqref{tFF}, we obtain the following consequence
\begin{equation}\label{tFF-xi}
\begin{split}
2\g(\tn_x\xi,y)=&\ F(\f^2y,\xi,x)-F(\xi, \f^2y,x)+\left\{F(\f y,\xi,\xi)+F(\xi,\f y,\xi)\right\}\eta(x)\\[4pt]
&+F(x,\f y,\xi)+F(\f^2y, \f x, \xi)+F(x,\f^2y,\xi).
\end{split}
\end{equation}
In deriving the last equality we have used the properties in \eqref{strM}.
We then apply the expression of $\f^2$ from \eqref{strM} and some properties of $F$ in this case.
The first is $F(\xi,\xi,x)=0$, which is a consequence of \eqref{defSlF}, and the second is the general identity $F(x,\xi,\xi)=0$, which comes from \eqref{F-prop1}. Thus, the relation in \eqref{tFF-xi} simplifies to the form
\begin{equation}\label{tFF-xi2}
2\g(\tn_x\xi, y)= F(\xi,x,y)+F(x,\f y,\xi)-F(y, \f x, \xi)-F(x,y,\xi)-F(y,x,\xi).
\end{equation}
Thereafter, we compute the various components in the above formula by exploiting the fact that the given manifold is Sasaki-like, \ie \eqref{defSlF} is valid, and we get:
\[
F(\xi,x,y)=0,\qquad
F(x,y,\xi)=g(\f x,\f y),\qquad
F(x,\f y,\xi)=-g(x,\f y).
\]
As a result, given the symmetry of $g(x,\f y)$ with respect to $x$ and $y$ as well as \eqref{gg}, the equality in
\eqref{tFF-xi2} simplifies to the following form
\begin{equation*}\label{tFF-xi3}
\g(\tn_x\xi, y)= -\g(\f x, y),
\end{equation*}
which is an equivalent expression of \eqref{tnxi}.
\end{proof}

Now we apply \eqref{tnxi} to \eqref{Ltg}  and use \eqref{gg} to get:
\begin{equation}\label{Ltg2}
\left(\LL_{\tvt}\g\right)(x,y) =\D{\tk}(x)\eta(y) + \D{\tk}(y)\eta(x) -2\tk\, g(\f x,\f y).
\end{equation}

Then we substitute the expression from \eqref{Ltg2} into \eqref{YS-tg} and obtain the following
\begin{equation}\label{YS-tg2}
  \frac12 \left\{\D{\tk}(x)\eta(y) + \D{\tk}(y)\eta(x)  \right\}-\tk\, g(\f x,\f y)= (\ttau - \tlm) \g(x,y).
\end{equation}
Contracting \eqref{YS-tg2}, we infer
\begin{equation}\label{YS-tg2tr}
  \D{\tk}(\xi)+2n\tk= \ttau - \tlm.
\end{equation}
On the other hand, we replace $x$ and $y$ in \eqref{YS-tg2} by $\xi$ and get
\begin{equation}\label{YS-tg2xx}
  \D{\tk}(\xi)= \ttau - \tlm.
\end{equation}
Then \eqref{YS-tg2tr} and \eqref{YS-tg2xx} imply
$
\tk=0,
$ 
which is not admissible for the potential and therefore the following holds
\begin{theorem}
There does not exist a Sasaki-like manifold $\M$ equipped with a $\g$-generated
Yamabe almost  soliton having a vertical potential.
\end{theorem}

\section{The case of a torse-forming vertical potential}

Let us recall,
a vector field $\vt$ on a (pseudo-)Riemannian manifold $(\MM,g)$ is called a \emph{torse-form\-ing vector field} if the following identity is true:
\begin{equation}\label{tf-v}
	\n_x \vt = f\,x + \gm(x)\vt,
\end{equation}
where $f$ is a differentiable function and $\gm$ is a 1-form \cite{Yano44,Sch54}.
The 1-form $\gm$ is called the \emph{generating form} and
the function $f$ is called the \emph{conformal scalar} of $\vt$ \cite{MihMih13}.

\begin{remark}\label{rem:types}
Some special types of torse-forming vector fields have been considered in various studies.
A vector field $\vartheta$ determined by \eqref{tf-v} is called respectively:
\begin{itemize}
	\item[-] \emph{torqued}, if $\gm(\vartheta) = 0$; \cite{Chen17}
	\item[-] \emph{concircular}, if $\gm = 0$; \cite{Yano40}
	\item[-] \emph{concurrent}, if $f - 1=\gm = 0$; \cite{YanoChen}
	\item[-] \emph{recurrent}, if $f = 0$; \cite{Wong61}
	\item[-] \emph{parallel}, if $f = \gm = 0$. (\eg \cite{Chen17a})
\end{itemize}
\end{remark}

If, in addition, the potential $\vt$ is vertical, \ie $\vt=k\,\xi$, then we get from \eqref{tf-v} the following
\begin{equation}\label{tf-vv}
	\D{k}(x)\xi+k\n_x \xi = f\,x + k\,\gm(x)\xi.
\end{equation}
Since $\eta(\n_x \xi)$ vanishes identically, \eqref{tf-vv} implies the following
\begin{equation*}\label{tf-vv2}
	\D{k}(x)= f\,\eta(x) + k\,\gm(x),
\end{equation*}
which, due to the nowhere-vanishing $k$, gives us the following expression for the generating form of $\vt$
\begin{equation}\label{tf-gm}
	\gm(x)=\frac{1}{k} \left\{\D{k}(x)- f\,\eta(x)\right\}.
\end{equation}
Then, the torse-forming vertical potential is determined by $f$ and $k$, hence we denote it by $\vt(f,k)$.

Plugging \eqref{tf-gm} into \eqref{tf-v}, we get
\begin{equation}\label{tf-v3}
	\n_x \vt = -f\,\f^2 x + \D{k}(x)\xi,
\end{equation}
which together with $\n_x \vt=\n_x (k\,\xi)=\D{k}(x)\xi+k\n_x\xi$ gives in the considered case the following
\begin{equation}\label{tf-nxi}
	\n_x \xi = -\frac{f}{k}\f^2 x.
\end{equation}

By virtue of \eqref{tf-nxi}, for the curvature tensor of $g$ we get
\begin{equation}\label{R-tf}
\begin{split}
R(x,y)\xi= -\left\{\D{h}(x)+h^2\eta(x)\right\}\f^2y+\left\{\D{h}(y)+h^2\eta(y)\right\}\f^2x,
\end{split}
\end{equation}
where we use the following shorter notation for the function which is the coefficient in \eqref{tf-nxi}
\begin{equation}\label{h}
h= \frac{f}{k}.
\end{equation}

As immediate consequences of \eqref{R-tf}, we obtain the following expressions
\begin{equation*}\label{rho-tf}
\begin{split}
R(\xi,y)z&= g(\f y,\f z)\grad{h} - \D{h}(z)\f^2 y +h^2\left\{\eta(z)y- g(y,z)\xi\right\},\\[4pt]
\rho(y,\xi)&= -(2n-1)\D{h}(y)-\left\{\D{h}(\xi)+2n h^2\right\}\eta(y),\\[4pt]
\rho(\xi,\xi)&= -2n\left\{\D{h}(\xi)+h^2\right\}.
\end{split}
\end{equation*}

Equality \eqref{tf-nxi}, due to \eqref{F-prop2} and \eqref{h}, can be rewrite in the form
\begin{equation}\label{tf-Fxi}
	F(x,\f y,\xi) = -h\,g(\f x,\f y).
\end{equation}
Bearing in mind $F(x,\xi,\xi)=0$, following from \eqref{F-prop1}, the expression \eqref{tf-Fxi} is equivalent to the following equality
\begin{equation}\label{tf-Fxi2}
	F(x,y,\xi) = -h\,g(x,\f y).
\end{equation}

Then \eqref{Lg} and \eqref{YS-g} imply
\begin{equation}\label{Lg-tf}
\begin{split}
\frac12\left\{\D{k}(x)\eta(y) + \D{k}(y)\eta(x)\right\} -f g(\f x, \f y)=(\tau-\lm)g(x,y).
\end{split}
\end{equation}
Contracting \eqref{Lg-tf} gives
\begin{equation}\label{Lg-tf-tr}
\begin{split}
\D{k}(\xi)+2n\,f=(2n+1)(\tau-\lm),
\end{split}
\end{equation}
and substituting $x=y=\xi$ into \eqref{Lg-tf} yields
\begin{equation}\label{Lg-tf-xx}
\begin{split}
\D{k}(\xi)=\tau-\lm.
\end{split}
\end{equation}
Then combining \eqref{Lg-tf-tr} and \eqref{Lg-tf-xx} leads to an expression for the conformal scalar of $\vt$ as follows
\begin{equation}\label{Lg-tf-f}
\begin{split}
f=\tau-\lm.
\end{split}
\end{equation}
This means that the following statement is valid
\begin{theorem}\label{thm:tau}
Let an accR manifold $\M$ be equipped with a Yamabe almost  soliton $(g;\vt(f,k),\lm)$,
where $\vt$ is a vertical torse-forming potential.
Then the scalar curvature $\tau$ of this manifold 
is the sum of the conformal scalar $f$
of $\vt$ and the soliton function $\lm$,
\ie
$
\tau=f+\lm.
$
\end{theorem}

By \eqref{Lg-tf-xx} and \eqref{Lg-tf-f} we have
\begin{equation}\label{f-dk}
f=\D{k}(\xi).
\end{equation}
Substituting \eqref{f-dk} into \eqref{h}, we obtain the following expression of the function $h$
\begin{equation*}\label{h-dk}
h=\D{(\ln k)}(\xi).
\end{equation*}

\begin{corollary}\label{cor:1}
The potential $\vt(f,k)$ of any Yamabe almost  soliton $(g;\vt,\lm)$ on $\M$ is a torqued vector field.
\end{corollary}
\begin{proof}
Due to \eqref{f-dk} and \eqref{tf-gm}, $\gm(\xi)$ vanishes.
Hence, $\gm(\vt)=0$ is true, \ie the potential $\vt$ is torqued given \remref{rem:types}.
\end{proof}

It is shown in \cite{Man62} that the class $\F_5$ is the only basic class in the considered classification of accR manifolds in which $\xi$ or its collinear vector field can be torse-forming.
Furthermore,
the general class of accR manifolds with a torse-forming $\xi$ is $\F_1 \oplus \F_2 \oplus \F_3 \oplus \F_5 \oplus \F_{10}$. Note that $\F_5$-manifolds are
counterparts of $\beta$-Kenmotsu manifolds in the case of almost contact metric manifolds.
The definition of the class $\F_5$ is given in the following way in \cite{GaMiGr}
\begin{equation}\label{defF5}
  F(x,y,z)=-\frac{\ta^*(\xi)}{2n}\left\{g(x,\f y)\eta(z)+g(x,\f z)\eta(y)\right\},
\end{equation}
where $\ta^*(\cdot)=g^{ij}F(e_i,\f e_j,\cdot)$ with respect to the basis $\{e_1,\dots,\A{}e_{2n},\A{}\xi\}$ of $T_p\MM$.
Moreover, on an $\F_5$-manifold the Lee form $\ta^*$ satisfies the property $\ta^*=\ta^*(\xi)\eta$.

Then in addition to the component in \eqref{tf-Fxi2} we have
\begin{equation}\label{F5}
  F(\xi,y,z)=0,\qquad \om=0.
\end{equation}

Let the potential $\tvt$ of the Yamabe almost  soliton $(\g;\tvt,\tlm)$ is also torse-forming and vertical, \ie
\begin{eqnarray*}
  &\tn_x \tvt = \tf\,x + \tgm(x)\tvt,\qquad 
  &\tvt=\tk\,\xi.\label{p2}
\end{eqnarray*}

Similarly, we obtain analogous equalities of \eqref{tf-v3} and \eqref{tf-nxi} for $\g$ and its Levi-Civita connection $\tn$ in the following form
\begin{eqnarray}\label{tf-v3-t}
	& \tn_x \tvt = -\tf\,\f^2 x + \D{\tk}(x)\xi,\\[4pt]
\label{tf-nxi-t}
	& \tn_x \xi = -\h\,\f^2 x,
\end{eqnarray}
where
\begin{equation*}\label{th}
\h=\frac{\tf}{\tk}.
\end{equation*}
Moreover, we also have  $\tf=\D{\tk}(\xi)$ and $\h=\D{(\ln \tk)}(\xi)$.

Thus, the following analogous assertions are valid.

\begin{theorem}\label{thm:ttau}
Let an accR manifold $\M$ be equipped with a Yamabe almost  soliton $(\g;\tvt(\tf,\tk),\tlm)$,
where $\tvt$ is a vertical torse-forming potential.
Then the scalar curvature $\ttt$ of this manifold 
is the sum of the conformal scalar $\tf$ of $\tvt$ and the soliton function $\tlm$, \ie
$
\ttt=\tf+\tlm.
$
\end{theorem}

\begin{corollary}\label{cor:2}
The potential $\tvt(\tf,\tk)$ of any Yamabe almost  soliton $(\g;\tvt,\tlm)$ on $\M$ is a torqued vector field.
\end{corollary}

The following equality is given in \cite{GaMiGr} and it expresses the relation between $\n$ and $\tn$ for the pair of B-metrics of an arbitrary accR manifold:
\begin{equation}\label{ntn}
\begin{split}
2g(\tn_x y,z)=&\ 2g(\n_x y,z)-F(x,y,\f z)-F(y,x,\f z)+F(\f z,x,y)\\[4pt]
&+\left\{F(y,z,\xi)+F(\f z, \f y, \xi)-\om(\f y)\eta(z)\right\}\eta(x)\\[4pt]
&+\left\{F(x,z,\xi)+F(\f z, \f x, \xi)-\om(\f x)\eta(z)\right\}\eta(y)\\[4pt]
&-\left\{F(\xi,x,y)-F(y,x,\xi)-F(x, \f y, \xi)\right.\\[4pt]
&\phantom{-\left\{F(\xi,x,y)\right.}\left.
-F(x,y,\xi)-F(y,\f x,\xi)\right\}\eta(z).
\end{split}
\end{equation}
By setting $y=\xi$, the last equality implies the following
\begin{equation}\label{ntn-xi}
\begin{split}
2g(\tn_x \xi,z)=&\ 2g(\n_x \xi,z)-F(x,\f z,\xi)-F(\xi,x,\f z)+F(\f z,x,\xi)\\[4pt]
&+\om(z)\eta(x)+F(x,z,\xi)+F(\f z, \f x, \xi).
\end{split}
\end{equation}

Taking into account \eqref{tf-nxi}, \eqref{tf-Fxi}, \eqref{tf-Fxi2} and \eqref{tf-nxi-t}, the relation \eqref{ntn} takes the form
\begin{equation*}\label{ntn-xi2}
2(\h-h)g(\f x,\f z)=F(\xi,x,\f z)-\eta(x)\om(z),
\end{equation*}
which for an $\F_5$-manifold, due to \eqref{F5}, implies
$
\h=h
$, \ie 
\begin{equation}\label{fk-fk2}
\frac{f}{k}=\frac{\tf}{\tk}.
\end{equation}

To express some curvature properties of accR manifolds, an associated quantity $\tau^*$ of the scalar curvature $\tau$ of $g$ is used in \cite{Man3}.
It is defined by the following trace of the Ricci tensor $\rho$: $\tau^*= g^{ij}\rho_{is}\f^s_j$ with respect to the basis $\{e_1,\dots,\A{}e_{2n},\A{}\xi\}$.
The relation between $\ttt$ and $\tau^*$ for a man\-i\-fold belonging to $\F_5^0\subset\F_5$ is given in \cite[Corolary 2]{Man3} as follows
\begin{equation}\label{tt-F5}
\ttt=-\tau^*-\frac{2n+1}{2n}\left(\ta^*(\xi)\right)^2-2\xi\!\left(\ta^*(\xi)\right).
\end{equation}
The subclass $\F_5^0$ of $\F_5$ is introduced in \cite{ManGri93} by the condition that the Lee form $\ta^*$ of the manifold be closed, \ie $\D\,\ta^*=0$. The last equality is equivalent to the following condition 
\begin{equation}\label{F50cond}
\D\!\left(\ta^*(\xi)\right)=\xi\!\left(\ta^*(\xi)\right)\eta.
\end{equation}

Using \eqref{tf-Fxi}, we compute that
\[
\ta^*(\xi)=2nh, \qquad \xi\!\left(\ta^*(\xi)\right)=2n\,\D{h}(\xi)
\]
and therefore \eqref{tt-F5} takes the form
\begin{equation}\label{tt-F5-}
\ttt=-\tau^*-2n(2n+1)h^2-4n\,\D{h}(\xi).
\end{equation}

%
%

\section{Example: The cone over a 2-dimensional complex space form with Norden metric}



In this section we consider an accR manifold construction given in \cite{HM16}.

First, let $(\NN,J,g')$ be a 2-dimensional almost complex manifold with Norden metric, \ie
$J$ is an almost complex structure 
and $g'$ is a pseudo-Riemannian metric with neutral signature such that $g'(Jx',Jy')=-g'(x',y')$ for arbitrary $x'$, $y'\in\Gamma(T\NN)$.
It is then known that $(\NN,J,g')$ is a complex space form with constant sectional curvature, denoted e.g. $k'$.

Second, let $\mathcal{C}(\NN)$ be the cone over $(\NN,J,g')$, \ie $\mathcal{C}(\NN)$ is the warped product $\R^+\times_t \NN$ with a generated metric $g$ as follows
\begin{equation*}\label{C-g}
 g\left(\left(x',a\ddt\right),\left(y',b\ddt\right)\right)
=t^2\,g'(x',y')+ab,
\end{equation*}
where $t$ is the coordinate on the set of positive reals $\R^+$ and $a$, $b$ are
differentiable functions on $\mathcal{C}(\NN)$.
Moreover, $\mathcal{C}(\NN)$ is equipped with an almost contact structure $(\f,\xi,\eta)$ by
\begin{equation}\label{C-str}
\f |_{\ker\eta}=J, \quad \xi=\ddt, \quad \eta=\D t, \quad \f\xi=0,\quad \eta\circ\f =0.
\end{equation}

Then $(\mathcal{C}(\NN), \f, \xi, \eta, g)$ is a 3-dimensional accR manifold belonging to the class $\F_1\oplus\F_5$. In particular, this manifold can be of $\F_5$ if and only if $J$ is parallel with respect to the Levi-Civita connection of $g'$, but the constructed manifold cannot belong to $\F_1$ nor to $\F_0$ \cite{HM16}.

Let the considered manifold $(\mathcal{C}(\NN), \f, \xi, \eta, g)$ belong to $\F_5$.
Using the result $\ta^*(\xi)=\frac{2}{t}$ from \cite{HM16}, we verify that the condition in \eqref{F50cond} holds and therefore $(\mathcal{C}(\NN), \f, \xi, \eta, g)$ belongs to $\F_5^0$.

Let $\left\{e_1,e_2,e_3\right\}$  be a basis in any tangent space at an arbitrary point of $\mathcal{C}(\NN)$ such that
\begin{equation}\label{fbasis}
\begin{array}{c}
\f e_1=e_2,\quad \f e_2=-e_1,\quad e_3=\xi,\\[4pt]
g(e_1,e_1)=-g(e_2,e_2)=g(e_3,e_3)=1,\quad
g(e_i,e_j)=0,\; i\neq j.
\end{array}
\end{equation}

In \cite{HM16} it is shown that the nonzero components of 
$R$ of the constructed 3-di\-men\-sional manifold with respect to the basis $\left\{e_1,e_2,e_3\right\}$
are determined by the equality
$R_{1212} =\frac{1}{t^2}(k' - 1)$ and the well-known properties of $R$.
Obviously, $(\mathcal{C}(\NN), \f, \xi, \eta, g)$ is flat if and only if $k'=1$ for $(\NN,J,g')$.
The nonzero components of the Ricci tensor of $(\mathcal{C}(\NN), \f, \xi, \eta, g)$ in the general case are then calculated as
$\rho_{11} = -\rho_{22} = \frac{1}{t^2}(k' - 1).$
Furthermore, the scalar curvature $\tau$ and the associated quantity $\tau^*$ of $(\mathcal{C}(\NN), \f, \xi, \eta, g)$ are given by
\begin{equation}\label{tau-k}
\tau = \frac{2}{t^2}(k' - 1), \qquad \tau^*=0.
\end{equation}

Then, taking into account the vanishing of $\tau^*$, the expression
\begin{equation}\label{t*t}
\ta^*(\xi)=\frac{2}{t}
\end{equation}
and $n=1$, we calculate $\ttt$ by \eqref{tt-F5} as
\begin{equation}\label{ttt}
\ttt = -\frac{2}{t^2}.
\end{equation}

Using the results $\n_{e_1}e_3=\frac{1}{t}e_1$, $\n_{e_2}e_3=\frac{1}{t}e_2$, $\n_{e_3}e_3=0$ from \cite{HM16} and $e_3=\xi$ from \eqref{fbasis}, we derive for any $x$ on $\mathcal{C}(\NN)$ the following formula
\begin{equation}\label{nxi-Ex}
\n_x \xi = -\frac{1}{t}\f^2 x.
\end{equation}
Comparing the last equality with \eqref{tf-nxi}, we conclude that
\begin{equation}\label{fkt}
\frac{f}{k}=\frac{1}{t},
\end{equation}
\ie $h=\frac{1}{t}$ holds due to \eqref{h} and then \eqref{tt-F5-} is also valid.

From \eqref{f-dk}, \eqref{fkt} and the expression of $\xi$ in \eqref{C-str}, we obtain the differential  equation $t\D{k}=k\D{t}$,
whose solution for the function $k(t)$ is
\begin{equation}\label{kct}
k = c t,
\end{equation}
where $c$ is an arbitrary constant.
Hence \eqref{fkt} and \eqref{kct} imply
\begin{equation}\label{f}
f = c.
\end{equation}

Taking into account \eqref{Lg}, \eqref{nxi-Ex} and \eqref{kct}, we obtain
\begin{equation}\label{Lg=2cg}
\LL_{\vt}g=2cg.
\end{equation}

Let us define the following differentiable function on $\mathcal{C}(\NN)$
\begin{equation}\label{lm}
\lm=\frac{2}{t^2}(k' - 1)-c.
\end{equation}
Then, bearing in mind \eqref{tau-k}, \eqref{Lg=2cg} and \eqref{lm}, we check that the condition in \eqref{YS-g} is satisfied and  $(g;\vt,\lm)$ is a Yamabe almost soliton with vertical potential $\vt$.

Due to \eqref{kk} and \eqref{kct}, the soliton potential  $\vt$ is determined by $\vt = c t \xi$.
Then, because of $\D{t}=\eta$ from \eqref{C-str} and \eqref{nxi-Ex}, we obtain
$\n_x\vt=c x$.  This means that $\vt$ is torse-forming with conformal scalar $f=c$ and
zero generating form $\gm$.
According to \remref{rem:types}, the torse-forming vector field $\vt$ is concircular in the general case of our example and in particular when $c=1$ it is concurrent.
Obviously, every concircular vector field is torqued, which supports \corref{cor:1}.

Taking into account \eqref{tau-k}, \eqref{f} and \eqref{lm}, we check the truthfulness of \thmref{thm:tau}.

In \cite{GNKG93}, a relation between the Levi-Civita connections $\n$ and $\tn$ of $g$ and $\g$, respectively, is given for $\F_5$ as follows
\[
\tn_x y=\n_x y-\frac{\ta^*(\xi)}{2n}\left\{g(x,\f y)+g(\f x,\f y)\right\}\xi.
\]
This relation for $(\mathcal{C}(\NN), \f, \xi, \eta, g)$ and $y=\xi$ implies
$
\tn_x \xi=\n_x \xi,
$
which due to \eqref{nxi-Ex} gives
\begin{equation}\label{tnxi-Ex}
\tn_x \xi = -\frac{1}{t}\f^2 x.
\end{equation}
The expression in \eqref{tnxi-Ex} follows also from \eqref{defF5}, \eqref{ntn-xi} and \eqref{t*t}.

Then, using \eqref{tf-nxi-t} and \eqref{tnxi-Ex}, we have
\begin{equation}\label{fkt2}
\frac{\tf}{\tk}=\frac{1}{t},
\end{equation}
which supports \eqref{fk-fk2} and \eqref{fkt}.

In a manner similar to obtaining \eqref{kct} and \eqref{f}, starting with \eqref{fkt2}, we find
\begin{gather}\label{ckt2}
\tk=\tilde{c} t,\qquad \tilde{c}=\mathrm{const},
\\[4pt]
\label{f2}
\tf =\tilde{c}.
\end{gather}

By virtue of \eqref{Ltg}, \eqref{tnxi-Ex} and \eqref{ckt2}, we have
\begin{equation}\label{Lg=2cg2}
\LL_{\tvt}\g=2\tilde{c}\g.
\end{equation}

We define the following differentiable function on $\mathcal{C}(\NN)$
\begin{equation}\label{lm2}
\tlm=-\frac{2}{t^2}-\tilde{c},
\end{equation}
which together with \eqref{ttt} and \eqref{Lg=2cg2} shows that the condition in \eqref{YS-tg} holds.
Then, $(\g;\tvt,\tlm)$ is a Yamabe almost soliton with vertical potential $\tvt$.

Using \eqref{tf-v3-t}, \eqref{ckt2}, \eqref{f2} and $\D t=\eta$ from \eqref{C-str}, we obtain
$\n_x\tvt=\tilde{c} x$, which shows that $\tvt$ is torse-forming with conformal scalar $\tf=\tilde{c}$ and
zero generating form $\tgm$.
Therefore $\tvt$ is concircular for arbitrary $\tilde{c}$ and concurrent for $\tilde{c} =1$.
Obviously, every concircular vector field is torqued, which supports \corref{cor:2}.
Furthermore,
the results in \eqref{ttt}, \eqref{f2}, and \eqref{lm2} support \thmref{thm:ttau}.


\end{document}